\documentclass[11pt]{amsart}
\usepackage{style}

\def\NZQ{\mathbb}               

\def\ZZ{{\NZQ Z}}

\begin{document}

\title{Volumes of Line Bundles on Schemes}
\author{Steven Dale Cutkosky}
\author{Roberto N\'u\~nez}

\thanks{The first author was partially supported by NSF grant DMS-1700046.}

\address{Steven Dale Cutkosky, Department of Mathematics,
University of Missouri, Columbia, MO 65211, USA}
\email{cutkoskys@missouri.edu}

\address{Roberto N\'u\~nez, Department of Mathematics,
University of Missouri, Columbia, MO 65211, USA}
\email{rnhvb@mail.missouri.edu}

\begin{abstract}
Volumes of line bundles are known to exist as limits on generically reduced projective schemes. However, it is not known if they always exist as limits on more general projective schemes. We show that they do always  exist as a limit on a codimension one subscheme of a nonsingular projective variety.
\end{abstract}


\maketitle

\section{Introduction}\label{Intro}

Suppose that $\mathscr L$ is a line bundle on a $d$-dimensional proper scheme $X$ over a field $k$. The volume of $\mathscr L$ is defined to be the lim sup
\begin{equation}\label{eqN1}
vol(\mathscr{L}) = \limsup \limits_{n \to \infty} \dfrac{\dim_k\Gamma(X,  \mathscr{L}^{n})}{n^{d}/d!}.
\end{equation}
This volume is defined in \cite[Definition 2.2.26]{Laz01}. Many important properties of the volume are derived in this book. The most fundamental question is if this volume exists as a limit. The volume does in fact exist as a limit in many cases. The following  theorem appears in \cite{Cut01}. 
Let $\mathcal N_X$ be the nilradical of a proper scheme $X$.

\begin{theorem}\label{Theorem14a}(\cite[Theorem 10.7]{Cut01}) Suppose that $X$ is a proper scheme of dimension $d$ over a  field $k$ such that $\dim \mathcal N_X<d$ and $\mathscr L$ is a  line bundle on $X$.  Then  the limit
$$
\lim_{n\rightarrow \infty}\frac{\dim_k \Gamma(X,\mathscr L^n)}{n^d}
$$
exists, and so $\mbox{vol}(\mathscr L)$ exists as a limit.
\end{theorem}
The condition on the nilradical just means that the scheme is reduced at all generic points of $d$-dimensional components. 

An example of a line bundle on a projective variety where the limit in Theorem \ref{Theorem14a} is an irrational number is given in Example 4 of Section 7 of \cite{CS}.

Theorem \ref{Theorem14a} is proven for line bundles on a nonsingular variety over an algebraically closed field of characteristic zero by Lazarsfeld (Example 11.4.7 \cite{Laz01}) using Fujita approximation (Fujita \cite{Fuj}). This result is extended by Takagi \cite{T} using De Jong's theory of alterations \cite{DJ} to hold on nonsingular varieties over algebraically fields of all characteristics $p\ge 0$.

Theorem \ref{Theorem14a} has been proven by  Okounkov  \cite{Ok} for section rings of ample line bundles,  for section rings of big line bundles and for graded linear series by Lazarsfeld and Musta\c{t}\u{a} \cite{LazMus}  and Kaveh and Khovanskii \cite{KavKov} when $k$ is an algebraically closed field. A local form of this result is given by Fulger in \cite{Ful}. These last proofs use an ingenious ``cone method'' introduced by Okounkov to reduce  to a problem of
counting points in an integral semigroup. All of these proofs require the assumption that {\it $k$ is  algebraically closed}.  The proof of Theorem \ref{Theorem14a} also uses this wonderful cone method.

The cone method applies to graded linear series. 
Suppose that $X$ is a  proper scheme over a field $k$. A graded linear series on $X$ is a graded $k$-subalgebra $L=\oplus_{n\ge 0}L_n$ of a section ring $\oplus_{n\ge 0}\Gamma(X,\mathscr L^n)$ of a line bundle $\mathscr L$ on $X$. 

   The following theorem is a consequence of \cite[Theorem 1.4]{Cut01} and \cite[Theorem 10.3]{Cut01}.
 
\begin{theorem}\label{TheoremI1} Suppose that $X$ is a $d$-dimensional projective scheme over a field $k$ with 
$d>0$. Let $\mathcal N_X$ be the nilradical of $X$. Then the following are equivalent
\begin{enumerate}
\item[1)] For every  graded linear series $L$ on $X$ there exists a positive integer $r$ (depending on $L$) such that the limit
$$
\lim_{n\rightarrow\infty}\frac{\dim_k L_{rn}}{n^d}
$$
exists.
\item[2)] $\dim \mathcal N_X<d$.
\end{enumerate}
\end{theorem}

The statement 1) $\Rightarrow$ 2) of Theorem \ref{TheoremI1} (\cite[Theorem 1.4]{Cut01}) is  established in \cite[Theorem 10.3]{CS}. It is shown that in  any proper $k$-scheme $X$  such that $\dim \mathcal N_X=d$ there exists a graded linear series $L$ which has the property that the limit of 1) does not exist (for any $r$). The construction in this proof  is  not a section ring of a line bundle. Thus we have the following interesting question.

\begin{question} Suppose that $X$ is a projective $d$-dimensional scheme over a field and $\mathscr L$ is a line bundle on $X$. Does the volume $\mbox{vol}(\mathscr L)$ exist as a limit?
\end{question}

This question has a positive answer if $\dim N_X<d$ by Theorem \ref{Theorem14a}.

In this paper we give a positive answer to this question for arbitrary codimension 1 subschemes of a nonsingular variety.

\begin{theorem} \label{thm: existence_limits_codimension1_in_Nonsingular}

Let $X$ be a nonsingular projective variety over an arbitrary field. Let $Y \subset X$ be a closed subscheme of codimension 1. Let $\mathscr{L}$ be an invertible sheaf on $Y$ and $d=\dim Y$. Then  the limit
$$
\lim_{n\rightarrow \infty}\frac{\dim_k \Gamma(Y,\mathscr L^n)}{n^d}
$$
exists, and so
$vol(\mathscr{L})$ exists as a limit.

\end{theorem}

Since the submission of this paper, N\'u\~nez \cite{N} has  proven that if $X$ is a projective variety whose nilradical $\mathcal N(X)$ satisfies $\mathcal N(X)^2=0$, then the volumes of all line bundles  on $X$ exist. 

\section{Notation and Terminology}

Let $X$ be a proper scheme of dimension $d$ over a field $\kappa$. Let $\mathscr{L}$ and $\mathcal{F}$ be an invertible sheaf and a coherent sheaf on $X$, respectively. We define the 
\textit{$\mathscr{L}$-volume of $\mathcal{F}$} to be

$$ vol_{\mathscr{L}}(\mathcal{F}) = \limsup \limits_{n \to \infty} \dfrac{h^{0}\left(X, \mathcal{F} \otimes \mathscr{L}^{n} \right)}{n^{d}/d!}$$
where $h^0(X,\mathcal F\otimes\mathscr L^n)=\dim_{\kappa}H^0(X,\mathcal F\otimes\mathscr L^n)$.

We say that the $\mathscr{L}$-volume of $\mathcal{F}$  exists as a limit if the limsup above is actually a limit. Observe that the 
$\mathscr{L}$-volume of $\mathcal O_X$ is the volume of $\mathscr{L}$.

 The volume of a graded linear series  $L$ (defined in Section \ref{Intro}) is:
$$vol(L) = \limsup \limits_{n \to \infty} \dfrac{\dim_{\kappa}L_n}{n^d/d!}. $$

The index of a graded linear series $L$ on a proper variety is defined to be  
\begin{equation}\label{eqN2}
m(L)=[\ZZ:G]
\end{equation}
 where $G$ is the subgroup of $\ZZ$ generated by $\{n\mid L_n\ne 0\}$.

The following theorem was proven in \cite{LazMus} for certain linear series on projective varieties over an algebraically closed field and in \cite{KavKov} for arbitrary linear series on projective varieties over algebraically closed fields.  The following theorem follows from \cite[Theorem 8.1]{CS}.

\begin{theorem}\label{KIthm}

Suppose that $X$ is a $d$-dimensional proper variety over a field $\kappa$, and $L$ is a graded linear series on $X$.
Let $m = m(L)$ be the index of $L$. Then

$$ \lim \limits_{n \to \infty} \dfrac{\dim_{\kappa}L_{mn}}{n^{d}}$$

\noindent exists.

\end{theorem}

 When $X$ is a variety, we define the rank of any coherent $\mathcal{O}_{X}$-module $\mathcal{F}$ to be the dimension of the $\mathcal{O}_{X,\eta}$-vector space $\mathcal{F}_{\eta}$, where $\eta$ is the generic point of $X$. 
 
 We will often call an invertible sheaf a line bundle.

\section{Preliminary Results} \label{sec: Preliminaries}

\subsection{Existence of Limits and Exact Sequences}

\begin{lemma}\label{lem: existence_limits_SES}

Let $X$ be a proper scheme of dimension $d$ over a field $\kappa$. Let $\mathscr{L}$ be an invertible sheaf on $X$. Let 

$$ 0 \rightarrow \mathcal{F}_{1} \rightarrow \mathcal{F}_{2} \rightarrow \mathcal{F}_{3} \rightarrow 0$$

\noindent be an exact sequence of coherent $\mathcal{O}_{X}$-modules. Let $\{1,2,3\} = \{i,j,k\}$. Suppose that the $\mathscr{L}$-volume of $\mathcal{F}_{i}$  exists as a limit and that $\mathcal{F}_{k}$ is supported on a closed subset of dimension strictly less than $\dim(X)$. Then the $\mathscr{L}$-volume of $\mathcal{F}_{j}$  exists as a limit as well. Moreover, 

$$ vol_{\mathscr{L}}(\mathcal{F}_{i})=vol_{\mathscr{L}}(\mathcal{F}_{j}) .$$

\end{lemma}

The proof of this lemma follows from taking cohomology of the short exact sequence tensored with $\mathscr L^n$, and the fact that if $\mathcal F$ is  a coherent sheaf whose support has dimension less than $d$, then the limit
$$
\lim_{n\rightarrow 0}\frac{h^i(X,\mathcal F\otimes\mathscr L^n)}{n^d}=0
$$
for all $i$ (by \cite[Proposition 1.31]{Deb} or \cite{K}).

\begin{corollary} \label{cor: existence_limits_isomorphism_away_from_closed_sets}

Let $X$ be a proper scheme over a field $\kappa$. Let $\mathscr{L}$ be an invertible sheaf on $X$. Let 

$$ 0 \rightarrow \mathcal{K} \rightarrow \mathcal{F} \rightarrow \mathcal{G} \rightarrow \mathcal{C} \rightarrow 0$$

\noindent be an exact sequence of coherent $\mathcal{O}_{X}$-modules. Suppose that $\mathcal{K}$ and $\mathcal{C}$ are supported on closed subsets of dimension strictly less than $\dim(X)$. Then the $\mathscr{L}$-volume of $\mathcal{F}$  exists as a limit if and only if the $\mathscr{L}$-volume of $\mathcal{G}$   exists as a limit and then 

$$ vol_{\mathscr{L}}(\mathcal{F})=vol_{\mathscr{L}}(\mathcal{G}) .$$

\end{corollary}

\begin{proof}

We break up the given exact sequence into the following two exact sequences:

$$0 \rightarrow \mathcal{K} \rightarrow \mathcal{F} \rightarrow \mathcal{F}/\mathcal{K} \rightarrow 0$$

\noindent and

$$0 \rightarrow \mathcal{F}/\mathcal{K} \rightarrow \mathcal{G} \rightarrow \mathcal{C} \rightarrow 0.$$

\noindent If the $\mathscr{L}$-volume of $\mathcal{F}$   exists as a limit, we apply \cref{lem: existence_limits_SES} to the first sequence to conclude that the same is true for $\mathcal{F}/\mathcal{K}$ and $vol_{\mathscr{L}}(\mathcal{F}/\mathcal{K})=vol_{\mathscr{L}}(\mathcal{F})$. Then, by \cref{lem: existence_limits_SES} and the second sequence, we see that the $\mathscr{L}$-volume of $\mathcal{G}$   exists as a limit and $vol_{\mathscr{L}}(\mathcal{F})=vol_{\mathscr{L}}(\mathcal{F}/\mathcal{K})=vol_{\mathscr{L}}(\mathcal{G})$.

An analogous argument shows that if the $\mathscr{L}$-volume of $\mathcal{G}$   exists as a limit, then the same is true for $\mathcal{F}$ and $vol_{\mathscr{L}}(\mathcal{F})=vol_{\mathscr{L}}(\mathcal{G})$.  

\end{proof}

\subsection{Non-zero-divisors, Invertible Sheaves, and Cartier Divisors}

\begin{lemma}\label{lem: existence_nzd}
Let $X$ be a projective scheme over a field.  Let $\mathscr{M}$ be an invertible sheaf on $X$ and let $\mathscr{A}$ be an ample invertible sheaf. Then, there exists an $n_{0} \in \mathbb{N}$ such that $H^{0}(X, \mathscr{A}^n \otimes \mathscr{M})$ contains a non-zero-divisor for $n \geq n_0$.
\end{lemma}

\begin{proof}

The ideal sheaf 0 of $\mathcal O_X$ has an irredundant primary decomposition $0=\mathcal{Q}_1\cap \cdots \cap \mathcal{Q}_r$ where the $\mathcal{Q}_i$ are primary ideal sheaves for some closed integral subvarieties $V_i$ of $X$. The $V_i$ are the associated subvarieties of $X$.  Such a decomposition can be found by sheafifying a homogeneous  irredundant primary decomposition of the zero ideal in any projective embedding of $X$.

Let $x_1,\ldots, x_s\in X$ be a set of distinct closed points such that each $V_i$ contains at least one of these points. Let $\mathcal I=\prod_{i=1}^s\mathcal I_{x_i}$, where $\mathcal I_{x_i}$ is the ideal sheaf of the point $x_i$. We have a short exact sequence of $\mathcal O_X$-modules 
$$
0\rightarrow \mathcal I\rightarrow \mathcal O_X\rightarrow \oplus_{i=1}^s\mathcal O_{x_i}\rightarrow 0.
$$

Tensor this exact sequence with $ \mathscr{A}^n \otimes \mathscr{M}$ and take cohomology. By Serre's Vanishing Theorem, for $n\gg 0$, we have an exact sequence
$$	
H^0(X, \mathscr{A}^n \otimes \mathscr{M})\rightarrow \oplus_{i=1}^sk(x_i)\rightarrow 0,
$$
where $k(x_i)$ is the residue field of the point $x_i$. Thus, there exists a section $s\in H^0(X,\mathscr{A}^n \otimes \mathscr{L})$ which does not vanish at any of the $x_i$. It follows that $s$ does not vanish at any of the $V_i$ and so $s$ is not a zero-divisor on $\mathcal O_X$.

\end{proof}

For any scheme $X$, the association $D \to \mathcal{O}_{X}(D)$ gives an injective homomorphism from the group of Cartier divisors modulo linear equivalence to $\Pic(X)$ (See \cite[Corollary II.6.14]{Har}). Nakai \cite{Nak} has shown that if X is a projective scheme over an infinite field then this homomorphism is an isomorphism. We deduce the known fact that this is also the case for projective schemes over arbitrary fields. 
 
\begin{corollary}\label{cor: invertible_sheaves_from_divisors}

Let $X$ be a projective scheme over a field. Then, for any invertible sheaf $\mathscr{L}$ on $X$, there exists a Cartier divisor $D$ such that $\mathscr{L} \simeq \mathcal{O}_{X}(D)$. Moreover, under the identification described above, $\Pic(X)$ is generated by effective divisors.

\end{corollary}

\begin{proof}

Choose an ample line bundle $\mathscr{A}$ on X. After perhaps replacing $\mathscr{A}$ with a positive power of itself, we can use \cref{lem: existence_nzd} with $\mathscr{M} = \mathcal{O}_{X}$ to find a non-zero-divisor $t \in H^{0}(X, \mathscr{A})$. Then $\mathscr{A} \simeq \mathcal{O}_{X}(H)$, where $H\coloneqq div(t)$.

Again by \cref{lem: existence_nzd}, for some $n \in \mathbb{N}$, we can find a non-zero-divisor
$$
s \in H^{0}(X, \mathcal{O}_{X}(nH) \otimes \mathscr{L}).
$$
 Thus, $\mathcal{O}_{X}(nH) \otimes \mathscr{L} \simeq \mathcal{O}_{X}(div(s))$. Setting $D = div(s) - nH$, we get that $\mathscr{L} \simeq \mathcal{O}_{X}(D)$.

For the last statement of the corollary, simply notice that $D$ is a difference of effective divisors.

\end{proof}

If $D$ and $E$ are Cartier divisors on a projective scheme $X$, we will write $D\sim E$ if the $\mathcal O_X$-modules $\mathcal O_X(D)$ and $\mathcal O_X(E)$ are isomorphic.

\subsection{A lemma on volume}

We now show that volumes are unaffected by tensoring with invertible sheaves.

\begin{lemma} \label{lem: tensoring_invertible_preserves_limit}

Let $X$ be a projective scheme over a field. Let $\mathscr{L}$ and $\mathscr{M}$ be invertible sheaves on $X$. Suppose that the volume of $\mathscr{L}$ exists as a limit. Then $ vol_{\mathscr{L}}(\mathscr{M}) $ also exists as a limit and

$$ vol_{\mathscr{L}}(\mathscr{M}) = vol(\mathscr{L}). $$

\end{lemma}

\begin{proof}

By \cref{{cor: invertible_sheaves_from_divisors}}, we can assume that $\mathscr{M}=\mathcal{O}_{X}(D)$ for some Cartier divisor $D$. Let us consider first the case where $D$ is effective. We have a short exact sequence 

$$ 0 \rightarrow \mathcal{O}_{X} \left( -D \right) \rightarrow \mathcal{O}_{X} \rightarrow \mathcal{O}_{D} \rightarrow 0 $$

\noindent which, after tensoring with $ \mathcal{O}_{X} \left( D \right) $, becomes

$$ 0 \rightarrow \mathcal{O}_{X} \rightarrow \mathcal{O}_{X} \left( D \right) \rightarrow \mathcal{O}_{D} \otimes \mathcal{O}_{X} \left( D \right) \rightarrow 0.$$

\noindent Now, the volume of $\mathscr{L} \otimes \mathcal{O}_{X}=\mathscr{L}$ exists by assumption. Moreover, the sheaf $\mathcal{O}_{D} \otimes \mathcal{O}_{X} \left( D \right)$ is supported on a proper closed set of $X$. Thus, by \cref{lem: existence_limits_SES}, the limit

$$\lim \limits_{n \to \infty} \dfrac{h^{0}\left(X,\mathcal{O}_{X} \left( D \right) \otimes \mathscr{L}^{n} \right)}{n^{d}} = \lim \limits_{n \to \infty} \dfrac{h^{0}\left(X,\mathscr{L}^{n} \right)}{n^{d}}$$

\noindent exists.

Suppose now that $D$ is an arbitrary Cartier divisor. Thanks to \cref{cor: invertible_sheaves_from_divisors}, we can write $D \sim A-B$ where both $A$ and $B$ are effective Cartier divisors. Since $B$ is effective, we have a short exact sequence

$$ 0 \rightarrow \mathcal{O}_{X} \left(-B\right) \rightarrow \mathcal{O}_{X} \rightarrow \mathcal{O}_{B} \rightarrow 0. $$

\noindent We can tensor the sequence above with $\mathcal{O}_{X} \left( A \right)$ and get 

$$ 0 \rightarrow \mathcal{O}_{X} \left(D\right) \rightarrow \mathcal{O}_{X} \left( A \right) \rightarrow \mathcal{O}_{B} \otimes \mathcal{O}_{X} \left(A \right) \rightarrow 0. $$

\noindent Again by \cref{lem: existence_limits_SES} we conclude that

$$ \lim \limits_{n \to \infty} \dfrac{h^{0}\left(X,\mathcal{O}_{X} \left( D \right) \otimes \mathscr{L}^{n} \right)}{n^{d}} = \lim \limits_{n \to \infty} \dfrac{h^{0}\left(X,\mathcal{O}_{X} \left( A \right)\otimes \mathscr{L}^{n} \right)}{n^{d}},   $$

\noindent where the limit on the right exists since $A$ is effective as proven above. Furthermore, 

$$\lim \limits_{n \to \infty} \dfrac{h^{0}\left(X,\mathcal{O}_{X} \left( A \right)\otimes \mathscr{L}^{n} \right)}{n^{d}} = \lim \limits_{n \to \infty} \dfrac{h^{0}\left(X, \mathscr{L}^{n} \right)}{n^{d}}. $$

\noindent Thus, $vol_{\mathscr{L}}(\mathscr{M})$ exists as a limit and $vol_{\mathscr{L}}(\mathscr{M})= vol(\mathscr{L})$ whenever $\mathscr{M}$ is invertible.

\end{proof}

\begin{remark} With the assumptions of the above lemma, we  have that $vol(\mathscr{L})$ exists as a limit if and only if 
$vol_{\mathscr{L}}(\mathscr{M})$ exists as a limit.
\end{remark}

\subsection{Ideal sheaves on nonsingular varieties}

The following result is essentially \cite[Lemma 13.8]{Cut_Book}. It is stated in \cite{Cut_Book} for nonsingular quasi-projective varieties over algebraically closed fields, but the proof given there carries over without modifications to the case of arbitrary fields.

\begin{lemma}[{{\cite[Lemma 13.8]{Cut_Book}}}] \label{lem: factorization_of_ideal_sheaves}

Suppose that $X$ is a quasi-projective nonsingular variety over an arbitrary field and $\mathcal{I}$ is an ideal sheaf on $X$. Then there exists an effective divisor $D$ on $X$ and an ideal sheaf $\mathcal{J}$ on $X$ such that the support of $\mathcal{O}_{X}/\mathcal{J}$ has codimension greater than or equal to 2 in $X$ and $\mathcal{I} = \mathcal{O}_{X}(-D)\mathcal{J}.$ 

\end{lemma}

\section{Nilradicals and Volumes} \label{sec: technical_proposition}

The following proposition is the main ingredient in our proof of 
\cref{thm: existence_limits_codimension1_in_Nonsingular}.

\begin{proposition} \label{prop: all_depends_on _volume_of_nilradical}

Let $X$ be an irreducible projective scheme over a field $\kappa$. Let $\mathcal{N}$ be the nilradical of $X$. Let $\mathscr{L}$ be an invertible sheaf on $X$. Suppose that $vol_{\mathscr{L}}(\mathcal{N} \otimes \mathscr{M})$ exists as a limit for any invertible sheaf $\mathscr{M}$. Then $vol(\mathscr{L})$ exists as a limit as well.
\end{proposition}

\begin{proof}

Let $\mathscr{A}$ be an ample invertible sheaf on $X$. We tensor the short exact sequence

\begin{equation}\label{eq: nilradical_SES}
    0 \rightarrow \mathcal{N} \rightarrow \mathcal{O}_{X} \xrightarrow{res} \mathcal{O}_{X_{red}} \rightarrow 0    
\end{equation}

\noindent with $\mathscr{A}$ to obtain

$$ 0 \rightarrow \mathcal{N} \otimes \mathscr{A} \rightarrow \mathscr{A} \xrightarrow{res} \mathscr{A} \otimes \mathcal{O}_{X_{red}} \rightarrow 0 .$$

\noindent By \cref{lem: existence_nzd}, after perhaps replacing $\mathcal{A}$ with a positive power of itself, we can find a section $\alpha \in H^{0}(X, \mathcal{A})$ such that $\alpha$ is not a zero-divisor. This condition is equivalent to the fact that $\alpha$ does not vanish along any associated subvariety of $X$. In particular, since $X_{red}$ is one of the associated subvarieties, we have that the restriction $\alpha|_{X_{red}}$ of $\alpha$ to $X_{red}$ is nonzero. The section $\alpha$ induces a short exact sequence

\begin{equation}\label{eq: hyperplane_SES}
0 \rightarrow \mathscr{A}^{-1} \xrightarrow{\alpha} \mathcal{O}_{X} \rightarrow \mathcal{O}_{H} \rightarrow 0,
\end{equation}

\noindent where $H=div(\alpha)$ is a $(d-1)$-dimensional scheme.

Similarly, after perhaps replacing $\mathscr{A}$ with a positive power of $\mathscr{A}$, we can assume that there is a section $\beta \in H^{0}(X,\mathscr{A} \otimes \mathscr{L})$ such that its restriction $\beta|_{X_{red}}$ to $X_{red}$ is nonzero.

For all $n \in \mathbb{N}^{+}$ we have restriction maps

$$ \Phi_{n}\colon H^{0}(X,\mathscr{L}^{n}) \rightarrow H^{0}(X_{red}, \left(\mathscr{L}|_{X_{red}}\right)^{n}). $$

\noindent Notice that the vector spaces $L_n \coloneqq \Phi_{n} \left( H^{0}(X,\mathscr{L}^{n}) \right)$ fit together into a linear series $L$ associated to $\mathscr{L}|_{X_{red}}$ on $X_{red}$.

We consider now the following two cases:

Case 1: Suppose that there exists $n_0 > 0$ such that the restriction map

$$ H^0(X, \mathscr{A}^{-1} \otimes \mathscr{L}^{n_0}) \rightarrow H^0(X_{red}, \mathscr{A}^{-1} \otimes \left(\mathscr{L}|_{X_{red}} \right)^{n_0} ) $$

\noindent is not zero. That is, there exists $\gamma \in H^0(X, \mathscr{A}^{-1} \otimes \mathscr{L}^{n_0})$ such that its restriction $\gamma|_{X_{red}}$ to $X_{red}$ is not zero. But then $\beta|_{X_{red}} \otimes \gamma|_{X_{red}}$ is a nonzero element of $L_{n_0+1}$ since $X_{red}$ is a variety. Moreover, for the same reason, $\alpha|_{X_{red}} \otimes \gamma|_{X_{red}}$ is a nonzero element of $L_{n_0}$ and this implies that the linear series $L$ has index $1$(the index of $L$ is defined in (\ref{eqN2})). 

By \cite[Theorem 8.1]{Cut01} (Theorem \ref{KIthm}), the limit 

$$ \lim \limits_{n \to \infty} \dfrac{\dim_{\kappa}L_n}{n^d} $$

\noindent exists.

After tensoring \eqref{eq: nilradical_SES} with $\mathscr{L}^n$, taking global sections, using the additivity of dimension, and dividing by $n^d$ we get that the limit

$$ \lim \limits_{n \to \infty} \dfrac{h^0(X, \mathscr{L}^n)}{n^d} = \lim \limits_{n \to \infty} \dfrac{\dim_{\kappa}L_n}{n^d} +  \lim \limits_{n \to \infty} \dfrac{h^{0}(X, \mathcal{N} \otimes \mathscr{L}^{n})}{n^d} $$

\noindent exists, since the limit on the right exists by taking $ \mathscr{M}=\mathcal{O}_{X} $ in our assumption. In particular, we observe that 

$$ vol(\mathscr{L}) = vol(L) + vol_{\mathscr{L}}(\mathcal{N}). $$

Case 2: Suppose now that for all $n>0$ the restriction map

$$ H^0(X, \mathscr{A}^{-1} \otimes \mathscr{L}^{n}) \rightarrow  H^0(X_{red},  \mathscr{A}^{-1} \otimes \left(\mathscr{L}|_{X_{red}} \right)^{n} ) $$

\noindent is the zero map. After tensoring \eqref{eq: nilradical_SES} with $\mathscr{A}^{-1} \otimes \mathscr{L}^{n}$ and taking global sections we see that this implies that

$$ H^0(X, \mathcal{N} \otimes \mathscr{A}^{-1} \otimes \mathscr{L}^{n} ) = H^0(X, \mathscr{A}^{-1} \otimes \mathscr{L}^{n}) $$

\noindent for all $n>0$. Since $vol_{\mathscr{L}}(\mathcal{N} \otimes \mathscr{A}^{-1})$ exists by assumption, we have that the limits

$$ \lim \limits_{n \to \infty} \dfrac{ h^0(X, \mathcal{N} \otimes \mathscr{A}^{-1} \otimes \mathscr{L}^{n} )}{n^d} = \lim \limits_{n \to \infty} \dfrac{h^0(X, \mathscr{A}^{-1} \otimes \mathscr{L}^{n})}{n^d} $$

\noindent exist.

Now we tensor the short exact sequence \eqref{eq: hyperplane_SES} with $\mathscr{L}^{n}$ to get

$$ 0 \rightarrow \mathscr{A}^{-1} \otimes \mathscr{L}^{n} \rightarrow \mathscr{L}^{n} \rightarrow \mathcal{O}_{H} \otimes \mathscr{L}^{n} \rightarrow 0.$$

\noindent Since the scheme $H$ is $(d-1)$-dimensional, by \cite[Proposition 1.31]{Deb}, $h^0(H, (\mathscr{L}|H)^{n})$ can grow at most as $n^{d-1}$. Thus, after taking global sections we see that the limit

$$ \lim \limits_{n \to \infty} \dfrac{h^0(X, \mathscr{L}^{n})}{n^d} = \lim \limits_{n \to \infty} \dfrac{h^0(X, \mathscr{A}^{-1} \otimes \mathscr{L}^{n})}{n^d}$$ 

\noindent exists and in this case

$$ vol(\mathscr{L}) = vol_{\mathscr{L}}(\mathcal{N} \otimes \mathscr{A}^{-1}) .$$

\end{proof}

\section{Proof of \cref{thm: existence_limits_codimension1_in_Nonsingular}}

As a first step towards proving \cref{thm: existence_limits_codimension1_in_Nonsingular}, we 
address the case where the subscheme $Y$ has an invertible ideal sheaf. 

\begin{proposition} \label{prop: existence_limits_pure_codimension_1}

Let $X$ be a nonsingular projective variety over an arbitrary field. Let $D$ be an effective divisor. Regard $D$ as a closed subscheme of $X$ with ideal sheaf $\mathcal{O}_{X}(-D)$ and let $\mathscr{L}$ be a invertible sheaf on $D$. Then the volume of $\mathscr{L}$ exists as a limit.

\end{proposition}

We begin by proving this result in the special case where $D$ has a single irreducible component.

\begin{proposition} \label{prop: exitance_limits_irreducible_codimension_1}

Let $X$ be a nonsingular projective variety over an arbitrary field and let $E$ be a prime divisor on $X$. For $m \in \mathbb{N}^{+}$, let $mE$ be the closed subscheme of $X$ with ideal sheaf $\mathcal{O}_{X}(-mE)$. Let $\mathscr{L}$ be an invertible sheaf on $mE$. Then, the volume of $\mathscr{L}$ exists as a limit.

\end{proposition}

\begin{proof}

We proceed by induction on $m$, where the case $m=1$ follows from \cite[Proposition 8.1]{Cut01} since $E$ is a projective variety. Fix some $m >1$. Let us begin by noticing that since $X$ is nonsingular, the nilradical $\mathcal{N}$ of the scheme $mE$ is equal to $\mathcal{O}_{X}(-E) \otimes \mathcal{O}_{(m-1)E} $. Thus, we can regard $\mathcal{N}$ as an invertible $\mathcal{O}_{(m-1)E}$-module. 

By \cref{prop: all_depends_on _volume_of_nilradical}, it is enough to prove that $vol_{\mathscr{L}}(\mathcal{N} \otimes \mathscr{M})$ exists as a limit for any invertible sheaf $\mathscr{M}$ on $mE$.  By induction, $vol(\mathscr{L}|_{(m-1)E})$ exists as a limit. Furthermore, the sheaf $\mathcal{N} \otimes \mathscr{M}$ is an invertible $\mathcal{O}_{(m-1)E}$-module. By \cref{lem: tensoring_invertible_preserves_limit}, 

$$vol_{\mathscr{L}|_{(m-1)E}}\left(\mathcal{N} \otimes \mathscr{M}\right) = vol(\mathscr{L}|_{(m-1)E})$$

\noindent and these volumes exist as limits. Now, for any $n \in \mathbb{N}$, since $\mathcal{N}$ is both an $\mathcal{O}_{mE}$-module and an $\mathcal{O}_{(m-1)E}$-module, we have that 

$$ H^{0}\left( mE, \mathcal{N} \otimes \mathscr{M} \otimes \mathscr{L}^{n} \right) = H^{0}\left((m-1)E, \mathcal{N} \otimes \mathscr{M} \otimes \left( \mathscr{L}|_{(m-1)E}\right)^n \right) $$

\noindent and hence the limit

$$vol_{\mathscr{L}}(\mathcal{N} \otimes \mathscr{M})= \lim \limits_{n \to \infty} \dfrac{H^{0}\left( mE, \mathcal{N} \otimes \mathscr{M} \otimes \mathscr{L}^{n} \right)}{n^d} = vol_{\mathscr{L}|_{(m-1)E}}\left(\mathcal{N} \otimes \mathscr{M}\right)$$

\noindent exists. Moreover, we see that

$$vol_{\mathscr{L}}(\mathcal{N} \otimes \mathscr{M})= vol(\mathscr{L}|_{(m-1)E}).$$

\end{proof}

\begin{proof}[Proof of \cref{prop: existence_limits_pure_codimension_1}]

Consider an arbitrary effective divisor $D=\sum_{i=1}^{t} a_i E_i$ on $X$. We have an exact sequence 

\begin{equation}\label{eqL1}
0 \rightarrow \mathcal{K} \rightarrow \mathcal{O}_{D} \rightarrow \displaystyle \bigoplus_{i=1}^{t} \mathcal{O}_{a_iE_i} \rightarrow \mathcal{C} \rightarrow 0.
\end{equation}

\noindent We claim that $\mathcal{K}=0$ and $C$ is supported in dimension at most $d-1$. Take $p \in D$ and let $f_i$ be a local equation for $E_i$ at $p$. The stalk at $p$ of the map $\mathcal{O}_{D} \rightarrow \oplus_i \mathcal{O}_{a_iE_i}$ is 

$$ \dfrac{\mathcal{O}_{X,p}}{(f_1^{a_1} \cdots f_t^{a_t})} \rightarrow \displaystyle \bigoplus_{i=1}^{t} \dfrac{\mathcal{O}_{X,p}}{(f_i^{a_i})},$$

\noindent which is an injection since $\mathcal{O}_{X,p}$ is a UFD and the $f_i$ are irreducible elements. Further, let $p \in E_j \setminus \cup_{i \neq j} E_{i}$. Then, except for $f_j$, all the $f_i$ become units in $\mathcal{O}_{X,p}$ and hence the map $\mathcal{O}_{D,p} \rightarrow \oplus_i \mathcal{O}_{a_iE_i,p}$ is an isomorphism. Thus, $supp(\mathcal{C}) \subset \cup_{i \neq j} (E_i \cap E_j)$. Each volume $vol(\mathscr{L}|_{a_iE_i})$ exists as a limit by \cref{prop: exitance_limits_irreducible_codimension_1} and so $vol(\mathscr{L} \otimes \left(\oplus_i\mathcal{O}_{a_iE_i} \right)) = \sum_{i} vol(\mathscr{L}|_{a_iE_i})$ exists as a limit. Thus we can apply \cref{lem: existence_limits_SES} to the exact sequence (\ref{eqL1}) to see that

$$ vol(\mathscr{L}) = \sum_{i=1}^{t} vol(\mathscr{L}|_{a_iE_i})$$

\noindent exists as a limit.

\end{proof}

We are now ready to give the proof of \cref{thm: existence_limits_codimension1_in_Nonsingular}. 
Recall that $X$ is a nonsingular projective variety over an arbitrary field and $Y$ is a closed subscheme of $X$ of codimension 1. We have that $\dim Y =d$ so that $\dim X=d+1$.

\begin{proof}[Proof of \cref{thm: existence_limits_codimension1_in_Nonsingular}]

Let $\mathcal{I}_{Y} \subset \mathcal{O}_{X}$ be the ideal sheaf of $Y$. By \cref{lem: factorization_of_ideal_sheaves}, $\mathcal{I}_{Y}=\mathcal{O}_{X}(-D)\mathcal{J}$ where $D>0$ and $codim\left(supp\left(\mathcal{O}_{X}/\mathcal{J} \right)\right) \geq 2$. The natural morphism of $\mathcal{O}_{Y}$-modules

$$ \dfrac{\mathcal{O}_{X}}{\mathcal{O}_{X}(-D)\mathcal{J}} \rightarrow \dfrac{\mathcal{O}_{X}}{\mathcal{O}_{X}(-D)} \oplus \dfrac{\mathcal{O}_{X}}{\mathcal{J}}$$

\noindent is an isomorphism away from the support of $\mathcal{O}_{X}/\mathcal{J}$. Moreover, we have that $vol_{\mathscr{L}}(\mathcal{O}_{X}/\mathcal{O}_{X}(-D)) = vol(\mathscr{L}|_{D})$ exists as a limit by \cref{prop: existence_limits_pure_codimension_1} and $vol_{\mathscr{L}}(\mathcal{O}_{X}/\mathcal{J}) = 0$ by \cite[Proposition 1.31]{Deb}. Thus, by \cref{cor: existence_limits_isomorphism_away_from_closed_sets}, the volume of $\mathscr{L}$ exists as a limit and

$$ vol(\mathscr{L}) = vol(\mathscr{L}|_{D}). $$

\end{proof}

\end{document}